\documentclass{amsart}

\usepackage{amssymb}
\usepackage{amsfonts,amsmath,amssymb,mathrsfs,verbatim}
\usepackage[dvips]{graphicx}
\usepackage{graphics}
\usepackage{psfrag}
\usepackage{hyperref}
\usepackage{marginnote}
\usepackage{xcolor}

\begin{document}

\renewcommand{\P}{\mathbb{P}}
\newcommand{\N}{\mathbb{N}}
\newcommand{\A}{\mathbb{A}}
\newcommand{\Z}{\mathbb{Z}}
\newcommand{\Q}{\mathbb{Q}}
\newcommand{\R}{\mathbb{R}}
\renewcommand{\d}{\underline{d}}
\newcommand{\lcm}{\mathrm{lcm}}
\newcommand{\mA}{\mathcal{A}}
\newcommand{\mB}{\mathcal{B}}
\newcommand{\mC}{\mathcal{C}}
\newcommand{\mH}{\mathcal{H}}
\newcommand{\mS}{\mathcal{S}}
\newcommand{\mP}{\mathcal{P}}
\newcommand{\mM}{\mathcal{M}}
\newcommand{\mK}{\mathcal{K}}
\newcommand{\mI}{\mathcal{I}}
\newcommand{\mU}{\mathcal{U}}
\newcommand{\mT}{\mathcal{T}}
\newcommand{\IA}{I_{\mA}}
\newcommand{\IB}{I_{\mB}}
\newcommand{\IS}{I_{\mS}}
\newcommand{\IM}{I_{\mM}}

\newcommand{\lS}{\leq_{\mS}}
\newcommand{\x}{\mathbf{x}}
\renewcommand{\t}{\mathbf{t}}
\newcommand{\kx}{k[\x]}
\newcommand{\kt}{k[\t]}
\newcommand{\mG}{\mathcal{G}}

\newtheorem{thm1}{Theorem}[section]
\newtheorem{lemma}[thm1]{Lemma}
\newtheorem{remark}[thm1]{Remark}
\newtheorem{definition}[thm1]{Definition}
\newtheorem{corollary}[thm1]{Corollary}
\newtheorem{proposition}[thm1]{Proposition}
\newtheorem{example}[thm1]{Example}
\newtheorem{problem}[thm1]{Open problem}
\newtheorem{question}[thm1]{Question}
\newcommand{\nacho}[1]{{\color{brown} \sf $\star\star$ Nacho: [#1]}}

\title{Toric ideals of graphs minimally generated by a Gr\"obner basis}
\author[Ignacio Garc\'ia-Marco]{Ignacio Garc\'ia-Marco}
\address{Ignacio Garc\'ia-Marco, Facultad de Ciencias and Instituto de Matem\'aticas y Aplicaciones (IMAULL), Secci\'on de
Matem\'aticas, Universidad de La Laguna, Apartado de Correos 456, 38200 La Laguna, Spain}
\email{iggarcia@ull.edu.es}
\author[Irene M\'arquez-Corbella]{Irene M\'arquez-Corbella}
\address{Irene M\'arquez-Corbella, Facultad de Ciencias and Instituto de Matem\'aticas y Aplicaciones (IMAULL), Secci\'on de
Matem\'aticas, Universidad de La Laguna, Apartado de Correos 456, 38200 La Laguna, Spain}
\email{imarquec@ull.edu.es}
\author[Christos Tatakis]{Christos Tatakis}
\address{Christos Tatakis, Department of Mathematics, University of Western Macedonia, 52100 Kastoria, Greece}
\email{chtatakis@uowm.gr}


\begin{abstract}
Describing families of ideals that are minimally generated by at least one, or by all, of their reduced Gr\"obner bases is a central topic in commutative algebra. In this paper, we address this problem in the context of toric ideals of graphs. We say that a graph $G$ is an MG-graph if its toric ideal $I_G$ is minimally generated by some Gr\"obner basis, and a UMG-graph if every reduced Gr\"obner basis of 
$I_G$ forms a minimal generating set. We prove that a graph $G$ is a UMG-graph if and only if its toric ideal $I_G$ is a generalized robust ideal (that is, its universal Gr\"obner basis coincides with its universal Markov basis). Although the class of MG-graphs is not closed under taking subgraphs, we prove that it is hereditary, that is, closed under taking induced subgraphs. In addition, we describe two families of bipartite MG-graphs: ring graphs (which correspond to complete intersection toric ideals, as shown by Gitler, Reyes, and Villarreal) and graphs in which all chordless cycles have the same length. The latter extends a result of Ohsugi and Hibi, which corresponds to graphs whose chordless cycles are all of length $4$.
\end{abstract}

\maketitle

\section{Introduction}

Let $R = \mathbb{K}[x_1,\ldots,x_m]$ be the polynomial ring in $m$ variables over a field $\mathbb{K}$. We say that an ideal $I \subseteq R$ is an {\it MG-ideal} if it is minimally generated by a Gr\"obner basis with respect to some monomial order, and a {\it UMG-ideal} if every reduced Gr\"obner basis of $I$ is a minimal generating set. Determining whether a given ideal $I$ is an MG-ideal or a UMG-ideal is, in general, a nontrivial task. The description of families of MG-ideals and UMG-ideals has become a central topic in commutative algebra.

The problem of describing MG-ideals has been tackled in \cite{CHT, HHHKR} among others; in these works the authors provide families of MG-ideals in several contexts. Furthermore, one common strategy for proving that the quotient ring $R/I$ is a Koszul algebra involves showing that $I$ is an MG-ideal generated by quadratic polynomials (see, e.g., \cite{CDR}). In the context of toric ideals, every normal toric ideal of codimension two is known to be an MG-ideal \cite{DHS}. Even in the one-dimensional case, characterizing MG-ideals is a challenging problem (see \cite[Open Problem 5.4]{Nacho-T1}). In this setting, it is known, for instance, that complete intersections MG-ideals correspond precisely to those arising from free numerical semigroups \cite[Theorem 4.7]{Nacho-T1}, and that numerical semigroups generated by arithmetic sequences also yield MG-ideals \cite{GSS}.

Clearly, the class of UMG-ideals is strictly contained within the class of MG-ideals, and it is not difficult to find examples of ideals that are MG-ideals but not UMG-ideals. A sufficient condition for an ideal $I$ to be a UMG-ideal is that it is robust, that is, minimally generated by its universal Gr\"obner basis. The notion of robust ideals was introduced by Boocher and Robeva in \cite{BOOCHER} and has since been extensively studied in the literature; see, e.g.,  \cite{BOOCHER2,BOOCHER3,Nacho-T1,KTV,KTV2,Sullivant,TATA}. As noted in the introduction of~\cite{BOOCHER}, one of the motivations to study robust ideals is that they are UMG-ideals. However, the family of UMG-ideals is strictly larger than that of robust ideals. For example, the only robust $1$-dimensional toric ideals are principal ideals, while every universally free numerical semigroup provides a UMG-ideal (see \cite{Nacho-T2}).  Interestingly, it is conjectured in \cite[Conjecture 3.19]{Nacho-T2} that all $1$-dimensional toric UMG-ideals are complete intersection, and hence they come from universally free numerical semigroups. For toric ideals, there is also the somehow related notion of generalized robustness, introduced by Tatakis in \cite{TATA} and later studied in \cite{Nacho-T1}. A toric ideal is generalized robust if its universal Gr\"obner basis is equal to the set of all minimal binomials (where a binomial is said to be minimal if it belongs to a minimal set of generators of the ideal).  

The goal of this paper is to study MG-ideals and UMG-ideals in the framework of toric ideals of graphs. The toric ideals of graphs were first considered by Villarreal in \cite{VILL1} and, since then, they have been an active topic of research. An interesting feature of these ideals is that many of their algebraic invariants can be described in terms of the underlying graph. Moreover, this family has turned out to be an interesting testing ground for more general problems and conjectures. We say that an undirected simple graph $G$ is an {\it MG-graph} (respectively a {\it UMG-graph}) if its corresponding toric ideal $I_G$ is an MG-ideal (respectively, a UMG-ideal). We prove that for toric ideals of graphs, the concepts of generalized robust ideals and UMG-ideal coincide, while they are not the same for toric ideals in general. 
Concerning MG-graphs, De Loera, Sturmfels and Thomas proved that complete graphs are MG-graphs (\cite[Theorem 2.1]{DELOERA}). In \cite{HIBI}, the authors provided an infinite family of (non-bipartite) non-MG-graphs. From the results of \cite{DELOERA} and \cite{HIBI} one deduces that the property of being an MG-graph is not preserved under taking subgraphs. Recently, in \cite{WL1,WL2}, the authors provided families of graphs that are not only MG-graphs, but also satisfy that all Betti numbers of the corresponding toric ideal and one of its initial ideals coincide. Interestingly, Ohsugi and Hibi \cite{Ohsugi} proved that for the toric ideal $I_G$ of a bipartite graph $G$, the following conditions are equivalent: (a) $G$ is bipartite chordal, (b) $I_G$ is generated by quadrics, and (c) $I_G$ has a Gr\"obner basis consisting of quadrics; and from the equivalence of (b) and (c) they also derived the following equivalent condition: (d) $R/I_G$ is Koszul. The main contribution of this result can be rephrased as: bipartite chordal graphs are MG-graphs. In the main result of this paper, we extend this result by proving the following.

\begin{thm1}\label{mainresult}
Let $G$ be a bipartite graph such that all minimal generators of $I_G$ have the same degree.  Then $I_G$ is an {\rm MG}-ideal. 
\end{thm1}

Moreover, we prove that whenever $G$ is a bipartite graph and $I_G$ is a complete intersection, then $G$ is an MG-graph; thus answering \cite[Open problem 5.2]{Nacho-T1} for bipartite graphs.

The present manuscript is organized as follows. 

In Section \ref{Section2}, we recall some fundamental facts related to the toric ideals, toric ideals of graphs, and their corresponding toric bases. 

Section \ref{Secion3} is devoted to the study of UMG-graphs. In Theorem \ref{thm:UMGgenrobust}, we show that $G$ is a UMG-graph if and only if $I_G$ is generalized robust. As a consequence, we establish that, for a bipartite graph $G$, the following four conditions are equivalent: (a) $G$ is chordless, (b) $I_G$ is robust, (c) $G$ is a UMG-graph, and (d) $I_G$ is generalized robust.

In Section \ref{Section4} we prove that being an MG-graph is preserved under taking induced subgraphs (Proposition \ref{induced-subgraphsMG}). We also present some constructions that produce MG-graphs (Proposition \ref{MG+UMG=MG}) and we prove that, in the bipartite case, complete intersection toric ideals always define MG-ideals (Corollary \ref{ci-MG ideal}).

Section \ref{Section5} contains the proof of our main result, Theorem \ref{mainresult}. To this end, we introduce in Definition \ref{varTheta} the notion of $\varTheta_r^k$ graphs, which play a key role in the characterization of graphs whose chordless cycles all have length $2k$, for $k\geq 3$ (see Theorem \ref{construction-thetas}). This structural result, together with Proposition \ref{MG+UMG=MG}, allows us to choose an appropriate monomial order to prove Theorem \ref{mainresult}. We also include examples that illustrate why this result cannot be extended straightforwardly to non-bipartite graphs. 

Finally, in Section \ref{sectionconclusion}, we conclude with a discussion of our results and pose several open questions for future research.

\section{Basic Notions on Toric Ideals and their toric bases}\label{Section2}

This section recalls basic notions concerning toric ideals of graphs and their corresponding toric bases.

Let $R = \mathbb K[x_1, \ldots, x_m]$ be a polynomial ring, where $\mathbb K$ is an arbitrary field. Given a monomial order $\succeq$ on $R$, and a polynomial $f\in R$, we denote by ${\rm in}_\succeq(f)$ -- or simply ${\rm in}(f)$ if no confusion arises -- the initial term of $f$ with respect to $\succeq$. Analogously, for an ideal $I \subseteq R$ we denote by ${\rm in}_\succeq(I)$ -- or by ${\rm in}(I)$ if no confusion arises -- the initial ideal of $I$ with respect to $\succeq$. For general facts and results about Gr\"obner bases we refer to \cite{CLO}.

Let $A = \{ {\bf a}_1, \ldots, {\bf a}_m \}\subseteq \mathbb N^n$ be a set of nonzero vectors, and let $\mathbb NA = \big\{ \sum_{i=1}^m b_i {\bf a}_i \mid b_i \in \mathbb N \big\} \subseteq \mathbb N^n$ denote the affine monoid spanned by $A$. We grade the polynomial ring $R$ by the monoid $\mathbb NA$ setting \( \deg_A(x_i) = \mathbf{a}_i \) for all \( i = 1, \ldots, m \). Then, for any ${\bf u} = (u_1, \ldots, u_m) \in \mathbb N^m$, the $A$-degree of the monomial ${\bf x^u} = x_1^{u_1}\cdots x_m^{u_m}$ is defined by $$\deg_A({\bf x^u}) = u_1 {\bf a}_1 + \cdots + u_m {\bf a}_m \in \mathbb NA.$$ The \emph{toric ideal} $I_A$ associated to $A$ is the binomial prime ideal defined as: 
$$I_A=\langle {\bf x^{\bf u}} - {\bf x^{\bf v}}\ :\ \deg_A({\bf x^{\bf u}}) = \deg_A ({\bf x^{\bf v}})\rangle.$$ We refer the reader to \cite{CLS, ES, MS, ST, VILL2} for a detailed study of toric ideals.

 Since $I_A$ is a binomial ideal, its reduced Gr\"obner basis with respect to any monomial order consists of binomials (see, e.g., \cite{ES}). The {\it universal Gr\"{o}bner basis} of $I_A$, denoted by $\mathcal{U}_A$, is the union of all reduced Gr\"obner bases $\mathcal{G}_{\succeq}$ of the ideal $I_A$ as $\succeq$ runs over all monomial orders. This set is finite (see, e.g., \cite[Theorem 1.2]{ST}) and consists entirely of binomials. 
  A {\it Markov basis} $M_A$ is a minimal binomial generating set of the toric ideal $I_A$ (its name Markov basis comes from its connection to Markov chains; see \cite[Theorem 3.1]{DST}).  The {\it universal Markov basis} of the ideal is denoted by $\mathcal{M}_A$ and is defined as the union of all the Markov bases of the ideal. As a consequence of the graded version of Nakayama's Lemma, $\mathcal M_A$ coincides with the set of binomials of $I_A$ not belonging to $\langle x_1,\ldots, x_m \rangle \cdot I_A$ (see \cite[Proposition 2.1]{Nacho-T1}).  Since $A \subseteq \N^n$, we have that $\mathbb N A$ is a pointed affine semigroup, and therefore, $\mathcal M_A$ is also a finite set (see \cite[Theorem 2.3]{THO1}). 
  We say that $I_A$ is robust if $\mathcal U_A$ minimally generates $I_A$; and it is generalized robust if $\mathcal U_A = \mathcal M_A.$ According to \cite[Corollary 5.12 $\&$ Corollary 5.13]{TATA}, a toric ideal is robust if and only if it is generalized robust and has a unique minimal generating set. 

Let $G$ be a finite and simple undirected graph with vertices $V(G)=\{v_1, \ldots, v_n\}$ and edges $E(G)=\{ e_1, \ldots, e_m\}$.
The toric ideal of $G$, denoted by $I_G$, is given by $I_G = I_{A_G}$, where $A_G = \{ {\bf a}_{e_1},\ldots, {\bf a}_{e_m}\} \subseteq \mathbb N^n$. Here, ${\bf a}_{e}$ is the characteristic vector of the edge $e$; that is, if $e = \{v_i,v_j\}$, then ${\bf a}_e = (0,\ldots,0,1,0, \ldots, 0,1,0, \ldots, 0)$ is the vector with ones at the $i$-th and $j$-th positions and zeros elsewhere. 

A {\it walk} $w=(u = u_0, u_1, \ldots, u_{\ell-1}, u_\ell = u')$ connecting $u \in V(G)$ and
$u' \in V(G)$ is a finite sequence of vertices of $G$,
such that $\{u_{j-1},u_{j}\}\in E(G)$, for $j=1,\ldots,\ell$. The {\it length}
of $w$ is the number $\ell$ of its edges, and we say that $w$ is an
{\it even} (respectively {\it odd}) walk if its length is even (respectively odd).
The walk $w$ is called {\it closed} if $u_{0}=u_{\ell}$, and is called a {\it path} if $u_{k}\neq u_{j},$ for every $1\leq k < j \leq \ell$. A {\it cycle}
is a closed path. 
A {\it chord} of $w$ is an edge of $G$ that joins two non-adjacent vertices of $w$. A walk $w$ is called {\it chordless} if it does not have chords.

Consider an even closed walk $w=(u_{0},u_{1},u_{2},\ldots,u_{2q-1},u_{2q} = u_0)$ of length $2q$ with $e_{i_j}=\{ u_{j-1},u_{j} \} \in E(G)$, for $j=1,\ldots,2q$. We write $$B_w =\prod _{k=1}^{q} x_{i_{2k-1}} - \prod _{k=1}^{q} x_{i_{2k}}.$$ Villarreal proved in \cite[Proposition 1]{VILL1} that
$$ I_G = \langle B_w  \ \vert \ w \text{ is an even closed walk}\rangle,$$ 
that is, the toric ideal $I_G$ is generated by the binomials corresponding to even closed walks of the graph $G$. 

Given a graph $G$, the universal Markov basis of $I_G$, denoted by $\mathcal M_G$, has been entirely described in \cite[Theorem 3.2]{TT1}; while the universal Gr\"obner basis of $I_G$, denoted by $\mathcal U_G$, has been described in \cite[Theorem 3.4]{THOMA}. 

Interestingly, for toric ideals of graphs, the inclusion $\mathcal M_G\subseteq \mathcal U_G$ holds (see \cite[Proposition 3.3]{TATA}), and in addition, $\mathcal U_G \subseteq \{ B_w  \ \vert \ w \text{ is an even closed walk}\}$. A complete description of the sets $\mathcal M_G$ and $\mathcal U_G$ is quite involved; we refer the reader to the aforementioned references for further details. In Proposition \ref{pr:bipbases}, we provide an explicit description of these sets in the particular case where $G$ is bipartite. 

\begin{proposition}\label{pr:bipbases}
    Let $G$ be a bipartite graph, then $I_G$ has a unique minimal set of generators, which is \[ \mathcal M_G = \{ B_c \, \vert \, c \text{ is a chordless cycle of G} \};\] and its universal Gr\"obner basis is \[ \mathcal U_G = \{ B_c \, \vert \, c \text{ is a cycle of G} \}.\]
\end{proposition}

\begin{example}{\rm
Consider the bipartite graph \( G \) shown in Figure~\ref{ex:bip}. It contains three even cycles:
\begin{itemize} 
\item $c_1 = (v_1,v_2,v_3,v_6,v_7,v_8,v_1),$
\item $c_2 = (v_3,v_4,v_5,v_6,v_3)$, and 
\item $c_3 = (v_1,v_2,v_3,v_4,v_5,v_6,v_7,v_8,v_1).$
\end{itemize}
The corresponding binomials are: $$B_{c_1} = x_1 x_7 x_9 - x_2 x_6 x_8, B_{c_2} = x_3 x_5 - x_4 x_9\ \text{and}\ B_{c_3} = x_1 x_3 x_5 x_7 - x_2 x_4 x_6 x_8.$$ 
We observe that $c_1$ and $c_2$ are chordless cycles. In contrast, since the edge $e_9 = \{v_3,v_6\}$ is a chord of $c_3$, the cycle $c_3$ is not chordless. Moreover, we have the identity $B_{c_3} = x_1 x_7 B_{c_2} + x_4 B_{c_1},$ so $B_{c_3} \in \langle x_1,\cdots,x_9 \rangle \cdot I_G$, and thus $B_{c_3}$ is not a minimal binomial.

Indeed, by Proposition \ref{pr:bipbases}, the universal Markov basis of $I_G$ is $\mathcal M_G = \{B_{c_1}, B_{c_2}\}$, and the universal Gr\"obner basis is  $\mathcal U_G = \{B_{c_1}, B_{c_2}, B_{c_3}\}$.
\begin{figure}[h]
\includegraphics[scale=1]{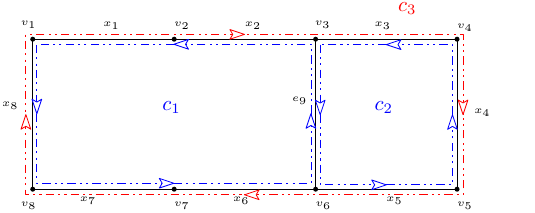}
\caption{Bipartite graph with three even cycles. The cycles $c_1$ and $c_2$ are chordless, while $c_3$ is not; indeed, the edge $e_9$ is a chord of $c_3$. }
\label{ex:bip}
\end{figure}
}
\end{example}

\section{UMG-graphs}\label{Secion3}

In general, UMG-ideals and generalized robust ideals are distinct concepts for toric ideals, as the following example shows.

\begin{example}\label{ex:UMGnotGenRob}{\rm Let $R=\mathbb{K}[x_1,x_2,x_3]$ be a polynomial ring over a field $\mathbb{K}$ and consider the $1$-dimensional toric ideal $I_A$ of $R$ with $A = \{10,15,42\} \subseteq \N$. The following four binomials belong to $I_A:$
\[f = x_1^3 - x_2^2, \hspace{.5cm} g = x_3^5 - x_1^{21}, \hspace{.5cm} h = x_3^5 - x_2^{14},   \hspace{.5cm}  p = x_3^5 - x_1^{18} x_2^{2}.\]
 By \cite[Theorem 4.2]{Nacho-T2} we have that the toric ideal $I_A$ is minimally generated by two binomials; indeed, $I_A = \langle f, g \rangle = \langle f, h \rangle = \langle f, p \rangle$ and, in particular, $f, g, h, p \in \mathcal M_A$. In addition, every reduced Gr\"obner basis of $I_A$ is either $\{f,g\}$ or $\{f,h\}$ and, as a consequence, $I_A$ is a {\rm UMG}-ideal. Nevertheless, $p \notin \mathcal U_A$, so $\mathcal U_A \subsetneq \mathcal M_A$ and, hence, $I_A$ is not generalized robust.
 }
\end{example}

However, the main result of this section establishes that these two notions are equivalent in the setting of toric ideals of graphs.

\begin{thm1}\label{thm:UMGgenrobust}
    A graph $G$ is an {\rm UMG}-graph if and only if $I_G$ is generalized robust.
\end{thm1}

In the proof of Theorem \ref{thm:UMGgenrobust} we use the combinatorial characterization of minimal binomials provided in \cite[Theorem 3.2]{TT1}. As the full statement of this result involves several technical notions, we include for convenience Lemma~\ref{lm:technical}, which extracts and summarizes only the ingredients of \cite[Theorem 3.2]{TT1} that are relevant to our context.

A {\it block} of a graph $G$ is a maximal connected subgraph of $G$ that does not contain any cut vertices. A block can consist of a single vertex, two vertices joined by an edge, or a biconnected graph.  

For an even closed walk $w = (z_0,\ldots,z_\ell = z_0)$, we identify $w$ with the subgraph with vertices $V(w) = \{z_0,\ldots,z_{\ell-1}\}$ and edges $E(w) = \{f_i  \, \vert \, 1 \leq i \leq \ell\},$ where $f_i = \{z_{i-1},z_{i}\}$. We say that a chord $e = \{z_i,z_j\}$ of $w$ is even (respectively odd) if it connects two vertices that are in the same block of $w$ and $j - i$ is odd (respectively even). We say that two odd chords $e = \{z_i,z_j\}$ and $e' = \{z_{i'}, z_{j'}\}$ with $0 \leq i < j < \ell$ and $0 \leq i' < j' < \ell$ {\it cross effectively} if $i' - i$ is odd (then necessarily $i' - j$, $j' - i$, $j' - j$ are all odd) and either $i < i' < j < j'$ or $i' < i < j' < j$. An $F_4$ of $w$ is a cycle of length four which consists of two edges $f_i,f_j$ of the walk $w$ with $i$ and $j$ of the same parity, and two odd chords $e$ and $e'$ that cross effectively in $w$.

\begin{lemma}\label{lm:technical}Let $w= (z_0,\ldots,z_\ell = z_0)$ be an even closed walk such that $B_w$ is a minimal binomial of $I_G$, then:
\begin{itemize}
\item all the chords of $w$ are odd, and
\item if two odd chords $e, e'$ cross effectively, then they form an $F_4$. 
 \end{itemize}
\end{lemma}

Let us illustrate this result with an example.

\begin{example}{\rm Consider the three even closed walks in Figure \ref{ex:nomin}, and let $G_i$ denote the subgraph induced by the vertex set $V(w_i)$, for each $i \in \{1,2,3\}$.

\begin{figure}[h]
\includegraphics[scale=.6]{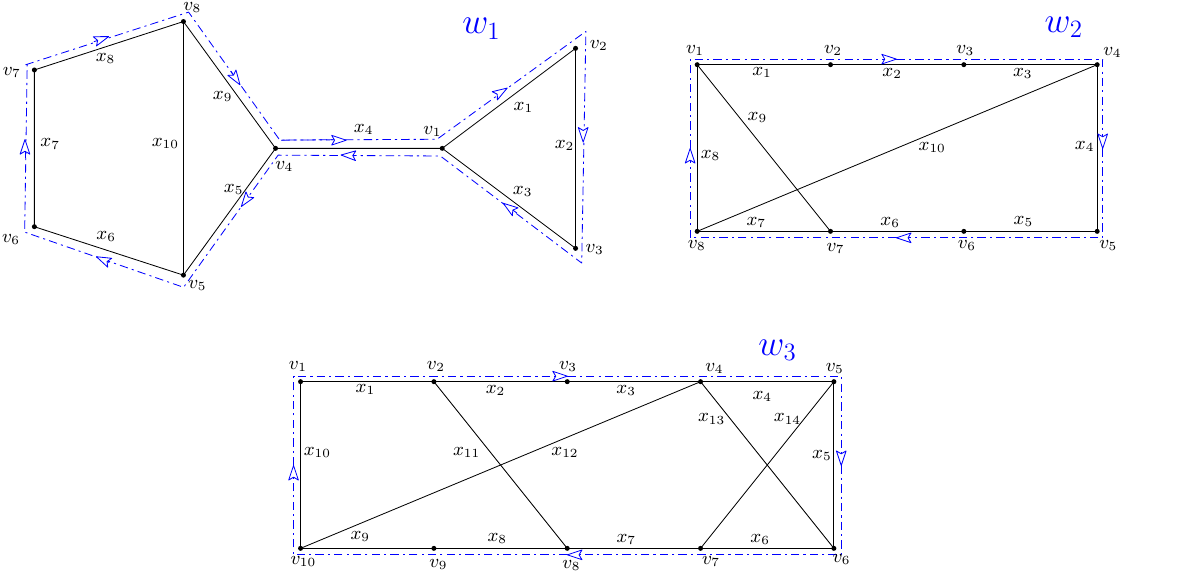}
\caption{The even closed walks $w_1, w_2$ yield non-minimal binomials $B_{w_1}, B_{w_2}$ by Lemma \ref{lm:technical}, while $B_{w_3}$ is minimal.}
\label{ex:nomin}
\end{figure}

In $G_1$, the even closed walk  $w_1 = (v_1,v_2,v_3,v_1,v_4,v_5,v_6,v_7,v_8,v_4,v_1)$  yields the binomial $B_{w_1} = x_1 x_3 x_5 x_7 x_9 - x_2 x_4^2 x_6 x_8$. The edge $e_{10} = \{v_5,v_8\}$ is an even chord of $w_1$; thus, by Lemma \ref{lm:technical}, $B_{w_1}$ is not a minimal binomial. Indeed, \[B_{w_1} = x_7  (x_1 x_3 x_5 x_9 - x_2 x_4^2 x_{10}) + x_2 x_4^2 (x_7 x_{10} - x_6 x_8),\]
with $x_1 x_3 x_5 x_9 - x_2 x_4^2 x_{10}, \, x_7 x_{10} - x_6 x_8 \in I_{G_1}$. Therefore, $B_{w_1} \in \langle x_1,\ldots,x_{10} \rangle \cdot  I_{G_1}$, and $B_{w_1} \notin \mathcal M_{G_1}.$

In $G_2$, the even cycle $w_2 = (v_1,v_2,v_3,v_4,v_5,v_6,v_7,v_8,v_1)$ yields the binomial $B_{w_2} = x_1 x_3 x_5 x_7 - x_2 x_4 x_6 x_8$. The edges $e_9 = \{v_1,v_7\}$ and $e_{10} = \{v_4,v_8\}$ are odd chords of $w_2$; they cross effectively and do not form an $F_4$. Hence, by Lemma \ref{lm:technical}, $B_{w_2}$ is not a minimal binomial. Indeed, \[B_{w_2} = x_5 (x_1 x_3 x_7 - x_2 x_9 x_{10}) + x_2 (x_5 x_9 x_{10} - x_4 x_6 x_8),\]
with $x_1 x_3 x_7 - x_2 x_9 x_{10}, x_5 x_9 x_{10} - x_4 x_6 x_8  \in I_{G_2}$, so $B_{w_2} \in \langle x_1,\ldots,x_{10} \rangle \cdot I_{G_3}$, and $B_{w_2} \notin \mathcal M_{G_2}.$

In $G_3$, the even cycle $w_3 = (v_1,v_2,v_3,v_4,v_5,v_6,v_7,v_8,v_9,v_{10},v_1)$ yields the binomial $B_{w_3} = x_1 x_3 x_5 x_7 x_9 - x_2 x_4 x_6 x_8 x_{10} \in I_{G_3}$. In this case, $B_{w_3} \in \mathcal M_{G_3}$, i.e., $B_{w_3}$ is a minimal binomial.  We observe that $w_3$ has four odd chords, namely the edges $e_{11} = \{v_2,v_8\},\, e_{12} = \{v_4,v_{10}\},$ $e_{13} = \{v_4,v_6\}$ and $e_{14} =  \{v_5,v_7\}$. The only pair of chords that cross effectively are $e_{13}$ and $e_{14}$, and they form an $F_4$. 
}
\end{example}

We are now ready to prove Theorem \ref{thm:UMGgenrobust}.

\begin{proof}[Proof of Theorem \ref{thm:UMGgenrobust}.]
    \noindent
\textbf{($\Longrightarrow$)} Suppose that $G$ is a UMG-graph. Then, by definition, every reduced Gr\"obner basis of $I_G$ is a minimal generating set, so in particular $\mathcal{U}_G \subseteq \mathcal{M}_G$. Moreover, for toric ideals of graphs, it is known that $\mathcal{M}_G \subseteq \mathcal{U}_G$; see \cite[Proposition~3.3]{TATA}. Hence, $\mathcal{U}_G = \mathcal{M}_G$, and therefore $I_G$ is generalized robust.

\medskip
\noindent
\textbf{($\Longleftarrow$)} Now suppose that $I_G$ is generalized robust, i.e., $\mathcal{U}_G = \mathcal{M}_G$. Let $\succeq$ be a monomial order and let $\mathcal G$ denote the reduced Gr\"obner basis of $I_G$ with respect to $\succeq$. Let $M \subseteq \mathcal G$ be a minimal set of generators of $I_G$. Our goal is to prove that $M = \mathcal G$. Assume, by contradiction, that there exists $f \in \mathcal G$ such that $f \notin M$. We write $f = x^\lambda - x^\mu$ with $\lambda, \mu \in \N^m$, and we have that \[ x^\lambda - x^\mu = \sum_{i = 1}^s x^{\delta_i} (x^{\alpha_i} - x^{\beta_i}), \]
    where: \begin{itemize}
        \item $\alpha_i, \beta_i, \delta_i \in \N^m$ for $i = 1,\ldots,s$,
        \item $x^{\lambda} = x^{\delta_1} x^{\alpha_1}$, $x^\mu = x^{\delta_s} x^{\beta_s},$
        \item $x^{\delta_i} x^{\beta_i} = x^{\delta_{i+1}} x^{\alpha_{i+1}}$ for $1 \leq i \leq s-1,$ and
        \item either $x^{\alpha_i} - x^{\beta_i} \in M$ or $x^{\beta_i} - x^{\alpha_i} \in M$;
    \end{itemize}   
    and we choose such an expression with the smallest value of $s$ possible.
    Since $I_G$ is generalized robust, it follows that $f \in \mathcal G \subseteq \mathcal U_G = \mathcal M_G$. Hence, there exists $t \in \{1,\ldots,s\}$ such that $x^{\delta_t} = 1$; otherwise $f \in \langle x_1,\ldots,x_n \rangle \cdot I_G$, contradicting that $f \in \mathcal M_G$. We now claim that $t = s$. Suppose instead that $t < s$. Then we may write:
    \[ x^\lambda - x^\mu = \underbrace{\sum_{i = 1}^{t-1} x^{\delta_i} (x^{\alpha_i} - x^{\beta_i})}_{h_1} + \underbrace{(x^{\alpha_t} - x^{\beta_t})}_{h_2} + \underbrace{\sum_{i = t+1}^s x^{\delta_i} (x^{\alpha_i} - x^{\beta_i})}_{h_3},  \]
    where $h_1 = x^\lambda - x^{\alpha_t}$ and $h_3 = x^{\beta_t}  - x^{\mu}$ are both nonzero.

    Since $f \in \mathcal G$,  we know that ${\rm in}(f) = x^\lambda$ and $x^\mu \notin {\rm in}(I_G)$. As $h_3 \in I_G$ and $x^\mu \notin {\rm in}(I_G)$, then 
    ${\rm in}(h_3) = x^{\beta_t}$ and, hence, the remainder of $x^{\beta_t}$ modulo $\mathcal G$ is $\overline{x^{\beta_t}}^{\mathcal G} = x^{\mu}$. Similarly, since either $h_2$ or $-h_2$ belongs to $\mathcal G$ and $x^{\beta_t} \in {\rm in}(I_G)$, it follows that $x^{\alpha_t} \notin {\rm in}(I_G)$, and $\overline{x^{\beta_t}}^{\mathcal G} = x^{\alpha_t}.$ Hence, $x^{\alpha_t} = x^\mu$, so $f = h_1$, contradicting the minimality of $s$.

    Therefore, $t = s$, and in particular: $$f = x^{\lambda} - x^{\mu}\ \text{and}\ g := x^{\alpha_s} - x^{\beta_s} = x^{\alpha_s} - x^{\mu}.$$
    Since $f,g \in \mathcal G$ we have that their leading terms ${\rm in}(f) = x^{\lambda}$ and ${\rm in}(g) = x^{\alpha_s}$ are minimal generators of ${\rm in}(I).$
    We now observe that $f - g = x^{\lambda} -  x^{\alpha_s} \in I_G$. Moreover, $\gcd(x^\lambda, x^{\alpha_s}) = 1$; otherwise, there exists $j \in \{1,\ldots,n\}$ such that $\frac{f - g}{x_j} =  \frac{x^\lambda}{x_j} - \frac{x^{\alpha_s}}{x_j} \in I_G$, which would contradict the minimality of $x^{\lambda}$ or $x^{\alpha_s}$ in ${\rm in}(I).$     

Now, since both $f$ and $g$ are in $\mathcal M_G$, there exist two even closed walks $w$ and $w'$ in $G$ such that $f = B_w = x^\lambda - x^\mu \in \mathcal M_G$ and $g = B_{w'} = x^{\alpha_s} - x^\mu \in \mathcal M_G$ and $\gcd(x^\lambda, x^{\alpha_s}) = 1$. We write $w = (z_0,z_1,\ldots,z_{2\ell} = z_0)$. Since $\gcd(x^\lambda, x^{\alpha_s}) = 1$, all the edges $f_i = \{z_i,z_{i+1}\} \in E(w)$ with $i$ even do not belong to $w'$. Hence, for each $i \in \{0,\ldots,2 \ell - 1\}$ there exists an edge $f_i' = \{z_i,z_j\} \in E(w') - E(w)$ connecting $z_i$ with another vertex of $V(w)$ different from $z_{i-1}$ and $z_{i+1}$; then $f_i'$ is a chord of $w$ and, by Lemma \ref{lm:technical},  $f_i'$ has to be an odd chord; so $i$ and $j$ have the same parity. Among all these chords, take $f_i' = \{z_i, z_j\}$ with $i < j$ such that the difference $j-i$ is the smallest possible. We separate two cases according to the parity of $i$ (and $j$).

\textit{Case 1:} $i,j$ are even. Then consider $f_{i+1}' = \{z_{i+1}, z_k\}$. The minimality of $j-i$ implies that $f_i'$ and $f_{i+1}'$ are two odd chords of $w$ that cross effectively. Since $B_w$ is minimal, Lemma \ref{lm:technical} implies that $f_i', f_{i+1}'$ form and $F_4$ in $w$, and thus $k =  j+1$. 
Therefore, there is an $F_4$ with edges $f_{i}, f_{i+1}', f_j, f_{j+1}'$ and $i,j$ even. Thus, there exists a binomial $h \in I_G$ of degree $2$, such that $h \notin \{f,g\}$ and whose initial form divides either $x^\mu$ or $x^{\alpha_s}$, contradicting the minimality of $f$ and $g$.

\textit{Case 2:} $i, j$ are odd. Then consider $f_{j-1}' = \{z_{j-1}, z_k\}$. Again, since $j-i$ is the smallest possible, then $f_i$ and $f_{j-1}'$ are two odd chords of $w$ that cross effectively. Since $B_w$ is minimal, by Lemma \ref{lm:technical},  both edges form an $F_4$ of $w$, and hence $k =  i-1$. Therefore, there is an $F_4$ with edges $f_{i-1}, f_{i}', f_{j-1}, f_{j-1}'$ and $i-1,j-1$ even. Thus, again there is a binomial $h \in I_G$ of degree $2$, such that $h \notin \{f,g\}$ and whose initial form divides either $x^\mu$ or $x^{\alpha_s}$, a contradiction.

In both cases, we arrive at a contradiction, so we must have  $\mathcal{G} = M$, and hence every reduced Gr\"obner basis of $I_G$ is a minimal generating set. Therefore, $G$ is a UMG-graph.

\end{proof}

 A graph $G$ is said to be \emph{chordless} if every cycle in $G$ is chordless. According to \cite[Corollary 3.5]{Nacho-T1}, a bipartite graph is chordless if and only if its toric ideal is generalized robust. Combining this characterization with the fact that $I_G$ admits a unique minimal generating set when $G$ is bipartite, and with the previous theorem, we obtain the following result.

\begin{corollary}\label{chordless-UMG} Let $G$ be a bipartite graph, the following are equivalent:
\begin{itemize}
\item $G$ is chordless.
\item $I_G$ is robust.
\item $G$ is a UMG-graph.
\item $I_G$ is generalized robust.
\end{itemize}
\end{corollary}

\section{MG-graphs}\label{Section4}

Given an ideal $I$ in a polynomial ring, determining whether it is an MG-ideal is, in general, a difficult problem. However, using \texttt{SageMath} \cite{sage}, one can algorithmically verify whether a given ideal $I$ is an MG-ideal by computing all of its reduced Gr\"obner bases. As an illustration, the following function takes as input a homogeneous (with respect to a positive grading, that is,  ${\rm deg}(x_i) > 0$ for all $i$)  ideal $I$ and returns $1$ if it is an MG-ideal or $0$ otherwise. The function works by comparing the minimal number of generators of the ideal with the minimal size among its reduced Gr\"obner bases.

{\small

 \begin{verbatim}
def check_mg_ideal(I):
    # Convert the ideal to Singular and compute minimum number of generators
    J_singular = singular(I)
    min_gen = len(J_singular.minbase())  # Number of generators in minbase
        
    # Compute the size of the smallest reduced GB via the Gröbner fan
    G = I.groebner_fan()
    L = G.reduced_groebner_bases()
    min_size_GB = min(len(basis) for basis in L)  

    # Compare and print the result
    if min_size_GB == min_gen:
        return (1)
    else:
        return (0)
\end{verbatim}
}

This function requires the computation of the whole Gr\"obner fan, and thus it can only handle small examples in a reasonable amount of time. We have used this function together with the \texttt{Nauty} library \cite{McKay} to check that the only bipartite graph with $\leq 8$ vertices that is not an MG-graph is the cube graph (the $1$-skeleton of the $3$-dimensional cube).

\begin{example} \label{First-BAD}{\rm
Consider the cube graph $Q_3$, see Figure \ref{counterexample}. Its toric ideal $I_{Q_3}$ has a unique minimal set of generators that consists of ten binomials:
\begin{eqnarray}
I_G & = &\langle x_1z_1-y_1y_4, x_2z_2-y_1y_2, x_3z_3-y_2y_3, x_4z_4-y_3y_4, x_1x_3-x_2x_4,  z_1z_3-z_2z_4,\nonumber \\
& & x_2y_3z_1-x_3y_1z_4, x_4y_2z_1-x_3y_4z_2, x_4y_1z_3-x_1y_3z_2, x_1y_2z_4-x_2y_4z_3 \rangle. \nonumber 
\end{eqnarray}

\begin{figure}[h]
\begin{center}
\includegraphics{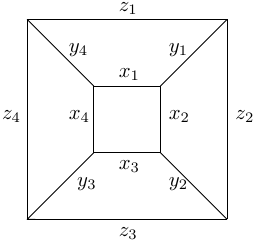}
\caption{$Q_3$ is a non MG-graph}
\label{counterexample}
\end{center}
\end{figure}

Every reduced Gr\"obner basis of $I_G$ has at least twelve elements, while the unique Markov basis of the ideal has ten elements. It follows that there is no monomial order such that the ideal $I_G$ is minimally generated by a Gr\"obner basis. 
}
\end{example}

Note that the property of being an MG-graph is not preserved under taking subgraphs, even within the class of bipartite graphs, as the following example shows.

\begin{example} {\rm
Consider the bipartite graph $H$ obtained by adding an edge to the graph $Q_3$ of Example \ref{First-BAD}; see Figure \ref{counterexample!}. The unique minimal set of generator of the toric ideal $I_H$ consists of twelve quadratic and one cubic binomials:
\begin{eqnarray}
\mathcal M_H & = \{ & x_1z_1-y_1y_4,\, x_2z_2-y_1y_2,\,  x_3z_3-y_2y_3,\,  x_4z_4-y_3y_4,\,  x_1x_3-x_2x_4,\,   \nonumber \\
& & z_1z_3-z_2z_4,\,  wy_3-x_4z_3,\,  wz_4-y_4z_3, \, wx_3-x_4y_2,\, wz_1-y_4z_2, \nonumber \\
& & wy_1-x_1z_2, \,  wx_2-x_1y_2,\,  x_2y_3z_1-x_3y_1z_4\ \}. \nonumber 
\end{eqnarray}

\begin{figure}[h]
\begin{center}
\includegraphics{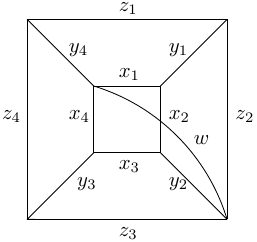}
\caption{Example of an MG-graph.}
\label{counterexample!}
\end{center}
\end{figure}

One has that $I_{H}$ has many reduced Gr\"obner basis that coincide with the Markov basis of $I_H$, i.e. $H$ is an MG-graph. In particular,
the reduced Gr\"obner basis of $I_G$ with respect to the lexicographic order with $x_3 > x_4 > y_2 > y_3 > z_2 > y_1 > x_1 > x_2 > z_1 > z_4 > w > y_4 > z_3$ equals $\mathcal M_G$.

}

\end{example}

Recall that an induced subgraph of $G = (V(G), E(G))$ is a subgraph $H=(V(H), E(H))$ of $G$ such that $V(H)\subseteq V(G)$ and the set $E(H)$ contains the edges of $G$ whose endpoints are both in $V(H)$. Although the MG-property is not preserved under taking subgraphs, in the next proposition we show that it is hereditary (i.e., preserved under taking induced subgraphs).

\begin{proposition}\label{induced-subgraphsMG}
Let $G$ be an {\rm MG}-graph, and let $H$ be an induced subgraph of $G$. Then $H$ is also an {\rm MG}-graph.
\end{proposition}

\begin{proof} Since $G$ is an MG-graph, there exists a monomial order $\succeq$ such that the reduced Gr\"obner basis of $I_G$ with respect to $\succeq$ is a minimal set of generators of $I_G$. Let us denote this basis by $\mathcal G= \left\{B_{w_1}, \ldots, B_{w_s} \right\}$ where each $w_1,\ldots,w_s$ is an even closed walk of $G$. We are going to prove that \[ \mathcal G_H := \mathcal G \cap \mathbb K[x_i \mid x_i \in E(H)] = \{B_{w_i} \, \vert \, V(w_i) \subseteq V(H), 1\leq i \leq s\} \] is both a minimal set of generators of $I_H$ and a reduced Gr\"obner basis for the monomial order $\succeq$ restricted to $\mathbb K[x_i \mid x_i \in E(H)]$.

By \cite[Theorem 4.13]{TT1}, we have that the property of being a minimal generator for $B_w$ only depends on the induced subgraph of $G$ with vertex set $V(w)$, and then $\mathcal G_H$ generates $I_H$. Thus, $\mathcal G_H$ is a minimal set of generators of $I_H$. 

Moreover, since $H$ is an induced subgraph of $G$, the ideal $I_G \cap K[x_i \, \vert \, x_i \in V(H)]$ corresponds (up to isomorphism) a combinatorial pure subring, and hence $\mathcal G_H$ is the reduced Gr\"obner basis of $I_H$ with respect to the monomial order restricted to $\mathbb K[x_i \mid x_i \in E(H)]$ (see \cite[Proposition 1.1]{OHH}). For completeness, we include a short proof of this. Let $f=x^\alpha - x^{\beta}\in I_H \subseteq I_G$, and suppose that $\mathrm{in}_{\succeq}(f) = x^{\alpha}$. Then, there exists $j\in \{1, \ldots, s\}$ such that $\mathrm{in}_{\succeq}(B_{w_j})$ divides $x^{\alpha}$. 

This implies that all variables in $B_{w_j}$ belong to $\mathbb{K}[x_i \mid x_i \in E(H)]$, and thus $V(w_j) \subseteq V(H)$. Therefore, $B_{w_j} \in \mathcal{G}_H$, and it follows that the leading terms of $\mathcal{G}_H$ generate $\mathrm{in}_{\succeq}(I_H)$. Hence,  $\mathcal{G}_H$ is a Gr\"obner basis of $I_H$. Finally, since $\mathcal{G}_H \subseteq \mathcal{G}$ and $\mathcal{G}$ is reduced, it follows that $\mathcal{G}_H$ is also reduced.
\end{proof}

A graph is an MG-graph (respectively, a UMG-graph) if and only if all its connected components are. In particular, the disjoint union of MG-graphs (or UMG-graphs) is again an MG-graph (or UMG-graph). 

In what follows, we explore how the MG and UMG properties behave under 1- and 2-clique sums. A {\it 1-clique sum} of two graphs $G_1$ and $G_2$ is formed from their disjoint union by identifying a vertex $v_1$ of $G_1$ and a vertex $v_2$ of $G_2$ to form a single shared vertex $v$ of the new graph. Similarly, a {\it 2-clique sum} of two graphs $G_1$ and $G_2$ is formed from their disjoint union by identifying an edge $e_1$ of $G_1$ and an edge $e_2$ of $G_2$ to form a single shared edge $e$ of the new graph; see Figure \ref{fig:clique-sum}.

\begin{figure}[h!]
\includegraphics[scale=0.7]{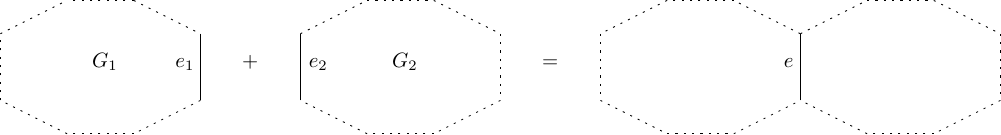}
\caption{A 2-clique sum of two graphs $G_1$ and $G_2$}
\label{fig:clique-sum}
\end{figure}

\begin{proposition}\label{pr:bipMG}
Let $G_1, G_2$ be two vertex-disjoint bipartite graphs, and let $K$ be a 1-clique sum of $G_1$ and $G_2$. Then $K$ is an {\rm MG}-graph if and only if both $G_1$ and $G_2$ are {\rm MG}-graphs.
\end{proposition}

\begin{proof} 
\noindent
\textbf{($\Longrightarrow$)} Since $G_1$ and $G_2$ are induced subgraphs of $K$, the result follows from Proposition \ref{induced-subgraphsMG}.

\vspace{0.5em}
\noindent
\textbf{($\Longleftarrow$)} Consider $I_{G_1} \subseteq \mathbb K[x_1,\ldots,x_m]$ and $I_{G_2} \subseteq K[y_1,\ldots,y_{m'}]$. Since $G_1$ (respectively $G_2$) is an MG-graph, there exists a monomial order $\succeq_1$ (respectively $\succeq_2$) in the polynomial ring $\mathbb K[x_1,\ldots,x_{m}]$ (respectively $\mathbb K[y_1,\ldots,y_{m'}]$) such that the reduced Gr\"obner basis of $I_{G_1}$ (respectively $I_{G_2}$) with respect to $\succeq_1$ (respectively $\succeq_2$), which we call $\mathcal G_{1}$ (respectively $\mathcal G_{2}$), minimally generates $I_{G_1}$ (respectively $I_{G_2}$).

As $K$ is bipartite and every cycle in $K$ is entirely contained in either $G_1$ or $G_2$, it follows from Proposition \ref{pr:bipbases} that $I_K = I_{G_1} + I_{G_2}$. In particular, the union $\mathcal G := \mathcal G_1 \cup \mathcal G_2$  minimally generates $I_K$. 

Define a monomial order $\succeq$ on $\mathbb{K}[x_1, \ldots, x_m, y_1, \ldots, y_{m'}]$ as the product order of $\succeq_1$ and $\succeq_2$. Since the variables in $\mathcal{G}_1$ and  $\mathcal{G}_2$ are disjoint, then $\mathcal{G}$ is the reduced Gr\"obner basis of $I_K$ with respect to $\succeq$. Hence, $K$ is an MG-graph.
\end{proof}

As a direct consequence of the previous result, we obtain the following.
\begin{corollary}\label{cor:bipMG}
Let $G$ be a bipartite graph with biconnected blocks $B_1,\ldots,B_r$. Then $G$ is an {\rm MG}-graph if and only if $B_1,\ldots,B_r$ are {\rm MG}-graphs.
\end{corollary}
\begin{proof} Every graph can be decomposed into its blocks via disjoint unions and 1-clique sums. If $G$ is not connected, the result follows from the fact that the disjoint union of MG-graphs is again an MG-graph. If $G$ is connected, then it can be reconstructed as a sequence of 1-clique sums of its biconnected blocks.

Since the toric ideal of a non-biconnected block (i.e., a single vertex or an edge) is the zero ideal, it trivially satisfies the MG-property. Therefore, the result follows from Proposition \ref{pr:bipMG}.
\end{proof}

\begin{proposition}\label{MG+UMG=MG}
Let $G_1, G_2$ be two vertex-disjoint graphs. If $G_1$ is an {\rm UMG}-graph, $G_2$ is an {\rm MG}-graph and at least one of them is bipartite; then the 1-clique sums and the 2-clique sums of $G_1$ and $G_2$ are {\rm MG}-graphs.
\end{proposition}

\begin{proof} Since the case of 1-clique sums follows analogously from Proposition \ref{pr:bipMG}, we focus on the case of 2-clique sums. Let us denote by $K$ the $2$-clique sum of the graphs $G_1$ and $G_2$ along the edge $e = \{u,v\}$.

Consider $I_{G_1} \subseteq \mathbb K[x_1,\ldots,x_m,e]$ and $I_{G_2} \subseteq K[y_1,\ldots,y_{m'},e]$, where $x_i\neq y_j$ for all $i\in\{1,\ldots m\}$ and $j\in\{1,\ldots,m'\}$. Let $\succeq_1$ be the degree reverse lexicographic order on $\mathbb{K}[x_1, \ldots, x_m, e]$ with $x_1 > \cdots > x_m > e$. Since $G_1$ is a UMG-graph, the reduced Gr\"obner basis of $I_{G_1}$ with respect to this order, which we call $\mathcal G_{1}$, minimally generates $I_{G_1}$. 

As $G_2$ is an MG-graph, there exists a monomial order $\succeq_2$ on $\mathbb K[y_1,\ldots,y_{m'},e]$ such that the reduced Gr\"obner basis $\mathcal G_2$ of $I_{G_2}$ with respect to $\succeq_2$ minimally generates $I_{G_2}$.

Since at least one of the graphs $G_1$ or $G_2$ is bipartite, it follows from \cite[Corollary 4.8]{VanTuyl} or \cite[Theorem 3.4]{GS} that $I_{K} = I_{G_1} + I_{G_2}$. Therefore, $\mathcal G := \mathcal G_1 \cup \mathcal G_2$  minimally generates $I_K$.

Now define a monomial order $\succeq$ on $\mathbb{K}[x_1, \ldots, x_m, y_1, \ldots, y_{m'}, e]$ as the product order of the restriction of $\succeq_1$ to $\mathbb K[x_1,\ldots,x_m]$ and $\succeq_2$ on $K[y_1,\ldots,y_{m'},e]$.  We claim that $\mathcal G$ is the reduced Gr\"obner basis of $I_K$ with respect to $\succeq$. 

Observe that for any homogeneous polynomial $f$, the following hold:
\begin{itemize}
    \item If $f \in \mathbb K[x_1,\ldots,x_m,e]$, then ${\rm in}_\succeq(f) = {\rm in}_{\succeq_1}(f)$. Moreover, if $e$ does not divide $f$, then ${\rm in}_{\succeq_1}(f) \in \mathbb K[x_1,\ldots,x_m]$.
    \item If $f \in \mathbb{K}[y_1, \ldots, y_{m'}, e]$, then ${\rm in}_\succeq(f) = {\rm in}_{\succeq_2}(f) \in \mathbb K[y_1,\ldots,y_{m'},e]$. 
\end{itemize}

To show that $\mathcal{G}$ is a Gr\"obner basis of $I_K$, we verify that the $S$-polynomial of any two distinct elements in $\mathcal{G}$ reduces to zero modulo $\mathcal{G}$, denoted $fSg \rightarrow_{\mathcal{G}} 0$. We distinguish three cases:

\begin{itemize}
    \item If $f, g \in \mathcal{G}_1$, then the reduction follows because $fSg \rightarrow_{\mathcal{G}_1} 0$, and $\mathrm{in}_{\succeq}(h) = \mathrm{in}_{\succeq_1}(h)$ for all $h \in \mathcal{G}_1$.
    
    \item If $f, g \in \mathcal{G}_2$, then $fSg \rightarrow_{\mathcal{G}_2} 0$, and $\mathrm{in}_{\succeq}(h) = \mathrm{in}_{\succeq_2}(h)$ for all $h \in \mathcal{G}_2$.
    
    \item If $f \in \mathcal{G}_1$ and $g \in \mathcal{G}_2$, then $\mathrm{in}_{\succeq}(f) \in \mathbb{K}[x_1, \ldots, x_m]$ and $\mathrm{in}_{\succeq}(g) \in \mathbb{K}[y_1, \ldots, y_{m'}, e]$, so their initial terms are relatively prime. Thus, $fSg \rightarrow_{\mathcal{G}} 0$.
\end{itemize}

Hence, $\mathcal G$ is a Gr\"obner basis of $I_K$  (see, e.g., \cite[Chapter 2, $\S$9]{CLO}) and since it is also minimal, $K$ is an MG-graph.
\end{proof}

Observe that if $G_1$ and $G_2$ are both bipartite UMG-graphs, then they are chordless, and so is their $1$-clique sum. Consequently, their 1-clique sum is also a UMG-graph by Corollary \ref{chordless-UMG}). However, their $2$-clique sum is not necessarily a UMG-graph. For instance, if $G_1$ and $G_2$ are both four-cycles, their corresponding toric ideals are principal and thus UMG-ideals. Nevertheless, any $2$-clique sum of such graphs yields a bipartite graph consisting of a six-cycle with a chord. Therefore, by Corollary \ref{chordless-UMG}, these graphs are not UMG-graphs.

We conclude this section by proving that toric ideals of bipartite graphs that are complete intersections are also MG-ideals. In other words, every complete intersection toric ideal coming from a bipartite graph has a complete intersection initial ideal. We obtain this result as an application of Proposition \ref{MG+UMG=MG} and the results of Gitler, Reyes and Villarreal in \cite{GRV}. In that work, the authors introduce the notion of a \emph{ring graph}: a graph $G$ is a ring graph if and only if all its biconnected blocks can be constructed by performing successive $2$-clique sums of cycles (see Figure \ref{fig:ringraphbip} for an example). In \cite[Corollary 3.4]{GRV}, they proved that for a bipartite graph $G$, the toric ideal $I_G$ is a complete intersection if and only if $G$ is a ring graph. Moreover, the toric ideal $I_c$ of an even cycle $c$ is a principal ideal, and thus it is a bipartite UMG-graph. Hence, iteratively applying Proposition \ref{MG+UMG=MG} we have that every block of a bipartite ring graph is an MG-graph. Finally, by Proposition \ref{pr:bipMG} we derive the following:

\begin{figure}[h]
\begin{center}
\includegraphics{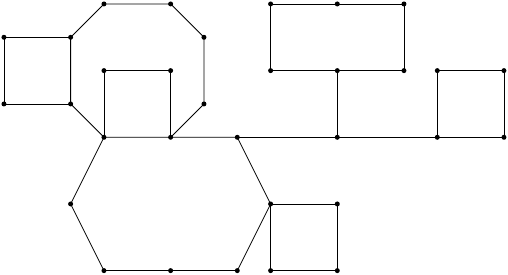}
\caption{Example of bipartite ring graph.}
\label{fig:ringraphbip}
\end{center}
\end{figure}

\begin{corollary}\label{ci-MG ideal}
Let $G$ be a bipartite graph. If $I_G$ is a complete intersection toric ideal, then it is an {\rm MG}-ideal.
\end{corollary}

The above result answers \cite[Open Problem 5.2]{Nacho-T1} for the toric ideals of bipartite graphs.

\section{Toric ideals of graphs generated in one degree}\label{Section5}

This section is devoted to generalizing the following result due to Ohsugi and Hibi:

\begin{thm1}\label{Koszul-bipartite}\cite{Ohsugi} Let $G$ be a bipartite graph such that the toric ideal $I_G$ is generated by quadrics. Then $I_G$ is an {\rm MG}-ideal.
\end{thm1}

Extending the above result, we prove the following: 

\begin{thm1}\label{Grobner-thetas}
Let $G$ be a bipartite graph such that all minimal generators of $I_G$ have the same degree. Then $I_G$ is an {\rm MG}-ideal. 
\end{thm1}

As we mention in Proposition \ref{pr:bipbases}, the toric ideal $I_G$ of a bipartite graph $G$ has a unique minimal binomial generating set. Moreover, its minimal binomial generators of degree $k$ are in bijection with even chordless cycles of length $2k$. Hence, all minimal generators of $I_G$ have the same degree $k$ if and only if all chordless cycles of $G$ have length $2k$. Taking into account these considerations, Theorem \ref{Grobner-thetas} can be restated as: Let $G$ be a graph such that every chordless cycle has length $2k$ (with $k \geq 2$), then $G$ is an MG-graph. 

To prove this result, we first provide in Theorem \ref{construction-thetas} a structural description of the graphs such that every chordless cycle has length $2k$ with $k \geq 3$. When $k = 2$ these graphs are known as bipartite chordal graphs and have been characterized in many ways (see, e.g., \cite{Huang}).

Firstly, we introduce the notion of the $\varTheta_r^k$ graph, which consists of two non-adjacent vertices and $r$ disjoint paths of length $k$ joining them. More precisely:

\begin{definition}\label{varTheta}
Let $r,k \geq 2$, we define $\varTheta_r^k$ as the graph consisting on $2$ vertices joined by $r$ vertex disjoint paths of length $k$. In other words, $\varTheta_r^k$ is the graph on the vertex set 
$$V(\varTheta_r^k)=\big\{ u_0,u_k \big\} \cup \big\{u_{i,j}\ : \ 1\leq i\leq r,\ 1\leq j\leq k-1\big\}$$ and on the edge set
$$E(\varTheta_r^k)=\big\{ \{u_0,u_{i,1}\},\ \{u_{i,k-1},u_k\},\ \{u_{i,j}, u_{i,j+1}\}\ : \ 1\leq i\leq r,\ 1\leq j\leq k-2 \big\}.$$ The vertices $u_0$ and $u_k$ are called the base points of the graph $\varTheta_r^k$.
\end{definition}

For all $r,k \geq 2$, the $\varTheta_r^k$ graph is a chordless graph where all cycles have length $2k$. In particular, a $\varTheta_2^k$ graph is a chordless cycle of length $2k$ . In Figure \ref{thetagraph}, we present the graph $\varTheta_4^5$. 

\begin{figure}[h]
\begin{center}
\includegraphics{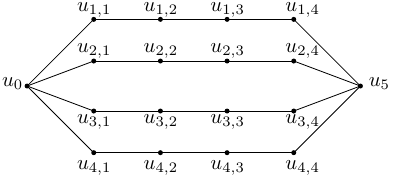}
\caption{The $\varTheta_4^5$ graph.}
\label{thetagraph}
\end{center}
\end{figure}

Note that by Proposition \ref{chordless-UMG} it follows that a graph $\varTheta_r^k$ is a UMG-graph for any $r,k\geq 2$. For the proof of Theorem \ref{construction-thetas}, we repeatedly use the following lemmas.

\begin{lemma} 
\label{lemma_cycles}
Let $G$ be a graph whose all chordless cycles have length $g$ and let $c$ be a cycle of $G$ of length $\ell$. Then  $$\ell\equiv2\mod(g-2).$$
\end{lemma}
\begin{proof}
Considering $c$ to be a cycle of $G$, we use induction on the length $\ell$ of $c$. If $\ell \leq g$, then $\ell = g$ and $\ell \equiv 2 \mod (g-2)$.  
 
 Now suppose that the cycle $c$ is of length $\ell > g$. Then $c$ is not chordless and there exists (at least) a chord that separates $c$ in two cycles $c_1$ and $c_2$ of lengths $\ell_1, \ell_2 < \ell$. By the induction hypothesis, we have that
$$\begin{array}{cc}
\ell_1 \equiv 2 \mod (g-2), &
\ell_2 \equiv 2 \mod (g-2).\end{array}$$
Since $\ell  = (\ell_1 -1) + (\ell_2-1)$, it follows
$$\ell = \ell_1 + \ell_2 - 2 \equiv 2 \mod (g-2),$$
 which proves the claim.
\end{proof}

\begin{lemma}
\label{lemma-Menger}
Let $G$ be a graph whose all chordless cycles have length $g \geq 5$. Let $H$ be a biconnected induced subgraph of $G$. If $v\in V(G)\setminus V(H)$, then $|N(v)\cap V(H)|\leq 1$.
\end{lemma}

\begin{proof}
Assume by contradiction that there exists a vertex $v \in V(G) \setminus V(H)$ such that $|N(v)\cap V(H)|\geq 2$. Then, we take $u,w\in V(H)$ two adjacent vertices to $v$ in $H$. Since $H$ is biconnected, according to Menger's Theorem there exist two disjoint paths in $H$ connecting $u$ and $w$. We denote by $\ell_1$ and $\ell_2$ the lengths of these paths (see Figure \ref{fig:menger}). 

\begin{figure}[h!]
\includegraphics[scale=0.7]{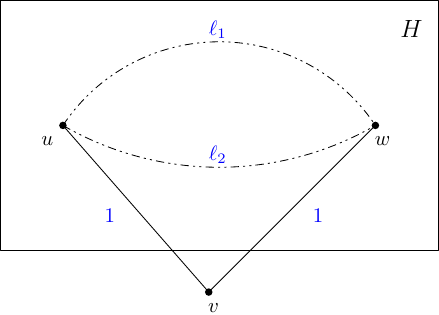}
\caption{The three paths between $u,w$ in Lemma \ref{lemma-Menger}.}
\label{fig:menger}
\end{figure}

Thus, there are three cycles in $G$ of lengths: $\ell_1+2$, $\ell_2+2$, and $\ell_1+\ell_2$, respectively. Applying Lemma \ref{lemma_cycles} we get that:$$ \begin{array}{rcc} \ell_1+2 & \equiv 2 \mod (g-2) &  (^{*}) \\ \ell_2+2 & \equiv 2 \mod (g-2) &  (^{**})\\ \ell_1+\ell_2 & \equiv 2 \mod (g-2) &  (^{***}) \end{array} $$
Combining $(^*)+(^{**})-(^{***})$ we get $4 \equiv 2 \mod (g-2)$, a contradiction.
\end{proof}

\begin{lemma}
\label{lemma-primitivecycles}
Let $G$ be a $2$-clique sum of two graphs $H$ and $K$. Then every cycle that passes through a vertex $u \in V(H) - V(K)$ and a vertex $v \in V(K) - V(H)$ is not chordless.
\end{lemma}

\begin{proof}
Assume that $G$ is the $2$-clique sum of $H$ and $K$ along the edge $e=\{w_1, w_2\}$. Any cycle passing through $u\in V(H) - V(K)$ and $v\in V(K)- V(H)$, necessarily passes through $w_1$ and $w_2$. Thus, $e$ is a chord of such a cycle and the cycle is not chordless.
\end{proof}

Next, we prove the structure of graphs whose all chordless cycles are of length $2k$ where $k\geq 3$. 

\begin{thm1}\label{construction-thetas} Let $k \geq 3$. All chordless cycles of a graph $G$ have length $2k$ if and only if all  biconnected blocks of $G$ are a 2-clique sum of $\varTheta_r^k$ graphs.
\end{thm1}

\begin{proof}It suffices to prove the result for biconnected graphs. 

\noindent $(\Longrightarrow)$ Assume that $G$ is a biconnected graph and that all its chordless cycles have length $2k$ with $k\geq 3$. Consider one of the largest (i.e., with most vertices) induced subgraph $H$ of $G$, which is a $2$-clique sum of $\varTheta_r^k$ graphs. Note that $H$ is not the empty graph because $G$ is biconnected, and thus there exists at least one chordless cycle in $G$ (or equivalently a $\varTheta_2^{k}$ graph). We aim to prove that $H = G$.

By contradiction, assume that $H \neq G$. Since $G$ is biconnected, there exists a path $\mathcal P$ of length $\ell \geq 2$ connecting two vertices $u, v \in V(H)$ such that $V(\mathcal P)  \cap V(H) = \{u,v\}$. Among all these paths, we choose the smallest one and let it be $\mathcal P = (w_0 = u, w_1, \ldots, w_{\ell-1}, w_{\ell} = v)$.  By Lemma \ref{lemma-Menger} we have $\ell \geq 3$.

\noindent {\it Claim:} Consider the induced subgraph $G'$ with vertices $V(\mathcal P) \cup V(H)$. Then,  
   $G'$ is a $2$-clique sum of $\varTheta_r^k$ graphs, for some $r\geq 2,k\geq 3$.

\begin{figure}[h]
\begin{center}
    \includegraphics[scale=.9]{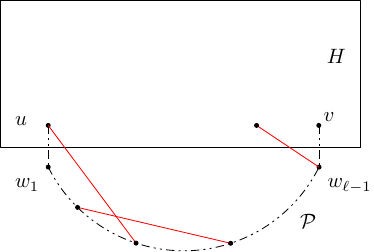}
\end{center}
\caption{Graph $G'$ in the proof of Theorem \ref{construction-thetas}. By the choice of $\mathcal P$ and Lemma \ref{lemma-Menger}, edges depicted in red cannot occur.}
\label{fig:Teo1}
\end{figure}

We first remark that $E(G') = E(\mathcal P) \cup E(H)$ (see Figure \ref{fig:Teo1}). Indeed, by the minimality of $\mathcal P$, one has that:
\begin{itemize}
\item[$\circ$] $\{x,w_i\} \notin E(G')$ with $2 \leq i \leq \ell-2$ and $x \in V(H)$, 
\item[$\circ$] $\{w_i,w_j\} \in E(G')$ with $0 \leq i < j \leq \ell$ if and only if $j = i+1$;
\end{itemize}
and, by Lemma \ref{lemma-Menger}, one has that:
\begin{itemize}
\item[$\circ$] $\{x,w_1\} \in E(G')$ with $x \in V(H)$ if and only if $x = u$, and 
\item[$\circ$] $\{x,w_{\ell-1}\} \in E(G')$ with $x \in V(H)$ if and only if $x = v$.
\end{itemize}

Now, since $u,v\in V(H)$ and the graph $H$ is biconnected, according to Menger's Theorem, there are at least two vertex disjoint paths in $H$ between $u$ and $v$. Among all these pairs of paths, we choose $\mathcal P_1$ of length $\ell_1$, and $\mathcal P_2$ of length $\ell_2$  such that $(\ell_1,\ell_2) \in \N^2$ is the smallest possible in lexicographical order.

Let $c_1 = (\mathcal{P}, -\mathcal{P}_1)$; that is, $c_1$ is the cycle obtained by concatenating $\mathcal{P}$ with~$- \mathcal P_1$, the reverse of the path $\mathcal{P}_1$. By the minimality of $\ell_1$ and the fact that $E(G') = E(\mathcal P) \cup E(H)$, we have that the cycle $c_1$ is chordless. Then, by hypothesis 
\[
    \ell+\ell_1 = 2k \ \ (^*) 
\]
Also, the cycles $c_2=\big( \mathcal P_1, - \mathcal P_2\big)$ and $c_3=\big( \mathcal P, - \mathcal P_2\big)$, of lengths $\ell_1+\ell_2$ and $\ell+\ell_2$ respectively, pass through $u$ and $v$ (see Figure \ref{fig:menger2}). By Lemma \ref{lemma_cycles} we have that
\[ \begin{array}{rcl}
\ell_1 + \ell_2 & \equiv 2 \mod (2k-2) &  (^{**})  \\
\ell + \ell_2 & \equiv 2 \mod (2k-2) &  (^{***})  
\end{array}\]

\begin{figure}[h!]
\includegraphics[scale=0.9]{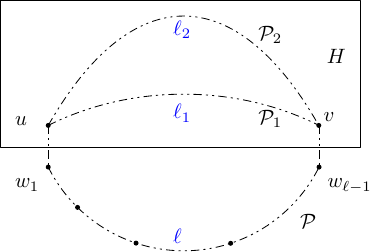}
\caption{Three paths connecting $u$ and $v$.}
\label{fig:menger2}
\end{figure}

Combining $(^{*}) - (^{**}) + (^{***})$ we get $2\ell  \equiv 2 \mod (2k-2)$. Since $\ell \leq 2k$, we have $\ell\in \{k, 2k-1\}$. We study these two cases separately.

\underline{\textbf{Case $\ell=2k-1$}}: By $(^{*})$ we get $\ell_1 =1$. Thus, $G'$ is a $2$-clique sum of $\varTheta_r^k$ graphs$;$ the graph $H$ which is a $2$-clique sum of $\varTheta_r^k$ graphs by hypothesis and the chordless cycle $c_1$ which is a $\varTheta_2^k$ graph. Thus, {\it Claim} is proved.

\underline{\textbf{Case $\ell=k$}}: By $(^{*})$ we get $\ell_1=k$. Let us see that $c_3$ is a chordless cycle. Indeed,\begin{itemize}
    \item[$\circ$] $\{u,v\} \notin E(G')$ because $\ell_1 > 1$, 
    \item[$\circ$] $c_3$ has no chords with both endpoints in $\mathcal P_2$, by the minimality of $\ell_2$, and
    \item[$\circ$] there are no other chords of $c_3$ because $E(G') = E(\mathcal P) \cup E(H).$
\end{itemize} 
Since $c_3$ has length $\ell+\ell_2$ and is chordless, then $\ell=\ell_1=\ell_2=k$. 

Now, using Lemma \ref{lemma-primitivecycles}, we have that $u$ and $v$ belong to the same component $K = \varTheta_r^k$. Otherwise, the cycle $c_2$ of length $\ell_1 + \ell_2 = 2k$ passing through $u$ and $v$ would not be chordless, a contradiction. To prove the {\it Claim} we are going to justify that the induced subgraph $K'$ with vertices $V(K') = V(K) \cup V(\mathcal P)$ is a $\varTheta_{r+1}^k$.

If $r=2$, then $K$ is just a cycle. Hence, $K'$ is a $\varTheta_3^k$ with base points the vertices $u,v$ (since $\ell+1=\ell_1=\ell_2=k$), and the {\it Claim} follows. Assume now that $r \geq 3$. Let us show that $u,v$ are the base points of $K$. Let $t$ be the distance between $v$ and the closest base point of $K$, which we denote by $w$; then $0 \leq t \leq k/2$. Let now $t'$ be the distance between $u'$ and the other base point of $K$, which we denote by $z$; then, $0 \leq t' < k$. 

\begin{figure}[h!]
\includegraphics[scale=.75]{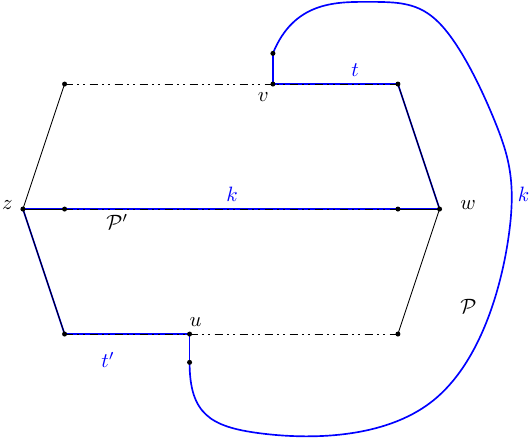}
\caption{Case $\ell=k$ in the proof of Theorem \ref{construction-thetas}. The chordless cycle $c$ of length $2k+t+t'$ described in the proof is depicted in blue.}
\label{Teo2}
\end{figure}

Since $r \geq 3$ there is a path $\mathcal P'$ of length $k$ between $w$ and $z$ that does not pass $u$ or $v$. Consider now $c$ the cycle consisting of the path $\mathcal P$, the shortest path in $K$ between $v$ and $w$ (of length $t\leq k/2$), let it be $\mathcal P'$, and the shortest path in $K$ between the vertex $z$ and $u$ (see Figure \ref{Teo2}). The length of $c$ is $2k+t+t'$ and, by Lemma \ref{lemma_cycles},
\[ 2 \equiv 2k + t + t' \equiv 2 + t + t' \mod (2k-2) ~~ \hbox{ with }t \leq k/2,\, t' < k.\]
 Thus, $t=t'=0$ or, equivalently, $u$ and $v$ are the base points of $K$. As a consequence, $K'$ is a $\varTheta_{r+1}^k$ graph and the {\it Claim} follows.

Thus, we conclude that $H = G$ and, hence, $G$ is a $2$-clique sum of $\varTheta_r^k$ graphs. 

\noindent $(\Longleftarrow)$ Conversely, consider $G$ a $2$-clique sum of $\varTheta_r^k$ graphs and let $c'$ be a chordless cycle of $G$. Since $c'$ is chordless, by Lemma \ref{lemma-primitivecycles}, it is contained in a $\varTheta_r^k$. The result follows from the fact that all the cycles in a $\varTheta_r^k$ have length $2k$ (Theorem \ref{construction-thetas}).

\end{proof}

Using the previous theorem, we can prove the main result of this manuscript.

\begin{proof}[Proof of Theorem \ref{Grobner-thetas}]
Let $G$ be a bipartite graph such that all minimal generators of $I_G$ have the same degree $k \geq 2$.
If $k = 2$ the result follows from Theorem \ref{Koszul-bipartite}.  If $k \geq 3$, then all chordless cycles of $G$ have length $2k$. 
Since $\varTheta_r^k$ is chordless for any $r,k\geq 2$, then $\varTheta_r^k$ is a UMG-graph by Corollary \ref{chordless-UMG}. The result follows from Proposition \ref{pr:bipMG}, Proposition \ref{MG+UMG=MG}, and Theorem \ref{construction-thetas}.
\end{proof}

When $G$ is bipartite, we have that all minimal generators of $I_G$ have the same degree if and only if all chordless cycles have the same length. If $G$ is nonbipartite, then the previous equivalence is no longer true. This suggests two ways of trying to extend Theorem \ref{Grobner-thetas} for nonbipartite graphs:
\begin{itemize}
    \item[(a)] If all chordless cycles of $G$ have the same length, then $G$ is an MG-graph.
    \item[(b)] If all the generators of $I_G$ have the same degree, then $G$ is an MG-graph.
\end{itemize}
We show examples justifying that none of these possible generalizations holds. 
In \cite[Example 2.1]{OhsugiHibi} the authors show a graph such that $I_G$ is generated by quadrics and has no quadratic Gr\"obner basis and therefore it is not an MG-graph. As a consequence, (b) does not hold. In Example \ref{ex:nobip1} we provide a counterexample to statement (a). Furthermore, in Example \ref{ex:nobip2} we exhibit a graph $G$ whose chordless cycles have all the same length, all the generators of $I_G$ have the same degree and $G$ is not an MG-graph. Both examples are chordal graphs, that is, graphs whose all chordless cycles have length $3$.

\begin{example} \label{ex:nobip1} {\rm On the left part of Figure \ref{fig:biconchordalnoquadnomg} we present a nonbipartite and biconnected chordal graph $G$ (graph number 35684 in {\tt houseofgraphs.org} \cite{housegraphs}). The corresponding toric ideal is minimally generated by nine generators; eight quadrics and one cubic:

\[ \begin{array}{lllcccc}
I_G & = &\langle & x_9 x_{11}- x_8 x_{12}, & x_1 x_{11}- x_2 x_{12}, &  x_7 x_{10} - x_8 x_{12}, \\ & &  & x_6 x_{10} - x_5 x_{12}, & x_3 x_{10} - x_4 x_{11},  & x_1 x_8 - x_2 x_9, \\ & & & x_5 x_7 - x_6 x_8, &  x_4 x_7 - x_3 x_9, &  x_1 x_3 x_5 - x_2 x_4 x_6 & \rangle.  

\end{array} \]

\begin{figure}[h!] 
\includegraphics[scale=0.65]{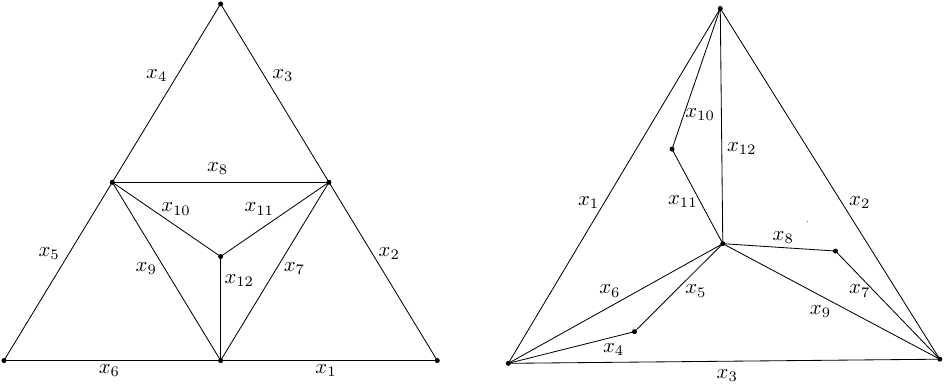} 
\caption{Biconnected chordal graphs that are not MG-graphs. The graph on the left corresponds to Example \ref{ex:nobip1}, its toric ideal is minimally generated by 8 quadrics and one cubic. The one on the right corresponds to Example \ref{ex:nobip2}, its toric ideal is minimally generated by 8 quadrics.}
\label{fig:biconchordalnoquadnomg}
\end{figure}

With SageMath, one can check that every Gr\"obner basis of $I_G$ at least 10 elements. Therefore, $G$ is not an MG-graph.
}
\end{example}

\begin{example} \label{ex:nobip2} {\rm
On the right part of Figure \ref{fig:biconchordalnoquadnomg} we present a chordal graph (graph number 36265 in {\tt houseofgraphs.org}), whose corresponding toric ideal is minimally generated by eight quadrics.

\[ \begin{array}{llrrrrrrr}
I_G & = &\langle &  x_9 x_{10}-x_2 x_{11}, & x_6 x_{10}-x_1 x_{11}, &  x_1 x_9-x_3 x_{12}, & x_2 x_8-x_7 x_{12}, \\ & &  & x_6 x_7-x_3 x_8, & x_2 x_6-x_3 x_{12},  & x_3 x_5-x_4 x_9, & x_1 x_5-x_4 x_{12} & \rangle.  \end{array} \]


One can check that every Gr\"obner basis of $I_G$ has at least 9 elements, which means that the graph is not an {\rm MG}-graph.

}
\end{example}

\section{Conclusion}\label{sectionconclusion}

The aim of this paper has been to study two natural classes of ideals:
\begin{itemize}
    \item UMG-ideals, i.e.,  ideals for which every reduced Gr\"obner basis is a minimal generating set;
    \item MG-ideals, i.e., ideals that admit at least one reduced Gr\"obner basis which is a minimal generating set.
\end{itemize}
We have explored these notions within the framework of toric ideals associated to graphs,

We have proved that the toric ideal of a graph is a UMG-ideal if and only if it is generalized robust (Theorem \ref{thm:UMGgenrobust}). Our proof of this equivalence relies heavily on the combinatorial structure of toric ideals of graphs. In particular, we observed that this equivalence does not hold in general for arbitrary toric ideals. For instance, there exist $1$-dimensional toric ideals that are UMG-ideals but not generalized robust. On the other hand, as a consequence of \cite{Nacho-T1} and \cite{Nacho-T2}, every $1$-dimensional toric ideal that is generalized robust is also a UMG-ideal. Whether this implication extends to all toric ideals remains an open question. Moreover, when $G$ is a bipartite graph, we proved in Corollary \ref{chordless-UMG} that the notions of chordless, robust, UMG and generalized robust all coincide for the ideal $I_G$.

Regarding MG-ideals, our main contributions focus on toric ideals of bipartite graphs. In Corollary \ref{ci-MG ideal}, we showed that if the toric ideal $I_G$ of a bipartite graph is a complete intersection, then it is an MG-ideal. This result, together with the description of bipartite graphs having complete intersection toric ideal by means of ring graphs \cite[Corollary 3.4]{GRV} yield that, for bipartite graphs, the following conditions are equivalent: 
\begin{itemize} 
\item[(a)] $I_G$  has a complete intersection initial ideal,
\item[(b)] $I_G$ is a complete intersection, and
\item[(c)] $G$ is a ring graph.  
\end{itemize}
This provides a partial answer to \cite[Open Problem 5.2]{Nacho-T1} when the graph is bipartite. In the non-bipartite case, we conjecture that conditions (a) and (b) are also equivalent. It is worth pointing out that the description of non-bipartite graphs with complete intersection toric ideal is significantly more complex (see, e.g., \cite{BGR, TT, TT2}).

In Theorem \ref{Grobner-thetas}, which is the main result of this manuscript, we have proved that whenever $G$ is a bipartite graph and $I_G$ is generated in one degree, then $I_G$ is an MG-ideal. This result extends one of Ohsugi and Hibi which corresponds to the case of generated by quadrics. Our proof of Theorem \ref{Grobner-thetas} is based on Theorem \ref{construction-thetas}, a combinatorial description of graphs whose chordless cycles have all length $2k$ for some $k \geq 3$. We discussed that two possible extensions of the result for nonbipartite graphs do not hold since there are chordal graphs providing counterexamples for both of them. Nevertheless, we do not know whether chordality is the only obstruction. More precisely, we do not know if the following holds: If all chordless cycles of $G$ have the same length $\ell \geq 4$, then $G$ is an MG-graph. 
Indeed, following similar ideas to the ones in Theorem \ref{construction-thetas}, one can prove the following structural result:
\begin{proposition}\label{construccion-thetas-impar} Let $k \geq 3$. All chordless cycles of a graph have length $2k-1$ if and only if all its biconnected blocks are a 2-clique sum of of chordless cycles of length $2k-1$
\end{proposition}

In Proposition \ref{induced-subgraphsMG}, we proved that the class of MG-graphs (i.e. those having an MG-ideal) is closed under taking induced subgraphs. This together with the fact that $Q_3$ is not an MG-graph (see Example \ref{First-BAD}) yields that every graph with an induced cube $Q_3$ is not an MG-graph. To the best of our knowledge, $Q_3$ is the only known example of a bipartite graph that is not an MG-graph. Characterizing MG-graphs in general remains an open problem. We believe that even the bipartite case presents significant challenges and merits further investigation.

\medskip

\noindent {\bf Acknowledgments}
\newline
This paper was written during the visit of the third author at the Department of Mathematics of the Universidad de La Laguna (ULL), whose hospitality is gratefully acknowledged.

The first and third authors are partially supported by the Spanish MICINN project ACOGE PID2022-137283NB-C22.

\end{document}